\documentclass[11pt,letterpaper]{article}
\usepackage{amsfonts, amsmath, amssymb, amscd, amsthm, graphicx,}
\usepackage{hyperref}
\newcommand{\e}{\varepsilon}
\renewcommand{\d}{{\rm d}}

\newtheorem{thm}{Theorem}[section]
\newtheorem*{ObraztsovTh}{Obraztsov's Theorem}
\newtheorem{cor}[thm]{Corollary}
\newtheorem{lem}[thm]{Lemma}
\newtheorem{prop}[thm]{Proposition}
\newtheorem{prob}[thm]{Problem}
\newtheorem{q}[thm]{Question}
\newtheorem{conj}[thm]{Conjecture}
\theoremstyle{definition}

\theoremstyle{remark}
\newtheorem{rem}[thm]{Remark}

\def\gp#1{\left\langle#1\right\rangle}
\def\nc#1{\gp{\!\gp{#1}\!}}
\def\Z{\mathbb Z}
\def\1{\{1\}}
\let\epsilon\varepsilon
\let\emptyset\varnothing

\begin{document}
\title{On topologizable and non-topologizable groups}
\author{Anton A. Klyachko,
\quad
Alexander Yu. Olshanskii,
\quad
Denis V. Osin
\thanks{This work was supported by the RFBR grant 11-01-00945.
The second and the third authors were also supported by
the NSF grants DMS-1161294 and DMS-1006345, respectively.}}
\date{}
\maketitle

\begin{abstract}
A group $G$ is called hereditarily non-topologizable if, for every
$H\le G$, no quotient of $H$ admits a non-discrete Hausdorff topology. We
construct first examples of infinite hereditarily non-topologizable
groups. This allows us to prove that $c$-compactness does not imply
compactness for topological groups. We also answer several other open
questions about $c$-compact groups asked by Dikranjan and Uspenskij. On
the other hand, we suggest a method of constructing topologizable groups
based on generic properties in the space of marked
$k$-generated groups. As an application, we show that there exist
non-discrete quasi-cyclic groups of finite exponent; this answers a
question of Morris and Obraztsov.
\end{abstract}

\renewcommand{\thefootnote}{\fnsymbol{footnote}}
\renewcommand{\thefootnote}{\arabic{footnote}}

\section{Introduction}

Throughout this paper, we always assume topological groups and spaces to
be Hausdorff. A well-known theorem of Kuratowski and Mrowka states that a
topological space $X$ is compact if and only if, for any topological space
$Y$, the projection $\pi _Y\colon X\times Y\to Y$ is closed. Motivated by
this theorem, Dikranjan and Uspenskij \cite{DU} call a topological group
$X$ \emph{categorically compact} (or \emph{$c$-compact} for brevity) if,
for every topological group $Y$, the image of every closed subgroup
of~$X\times Y$ under the projection $\pi _Y\colon X\times Y\to Y$ is
closed in $Y$.

Obviously every compact group is $c$-compact, while the converse was open
until now even for discrete groups. More precisely, the following
questions were asked in \cite[Question 1.2 and Question 5.2]{DU} (see also
\cite[Problem 31 (i) and Question 34]{OPT}).

\begin{prob}\label{prob1}
\hfil
\begin{enumerate}
\item[\rm(a)] Is every $c$-compact group compact?
\item[\rm(b)] Is every discrete $c$-compact group finite (finitely
generated, of finite exponent, countable)?
\end{enumerate}
\end{prob}

These questions have received considerable attention in the recent years. A complete survey of recent results can be found in the book \cite{Luk-book}, which is devoted to these problems. Until now, only results in the affirmative direction were known. For example, the answer to~(a) is known to be positive for solvable groups,
connected locally compact groups \cite{DU}, and maximally almost periodic groups \cite{Luk02}. Note also that every discrete $c$-compact group is necessarily a torsion group by
\cite[Theorem 5.3]{DU}.

Our first goal is to show that the answer to all parts of Problem
\ref{prob1} is, in fact, negative. Our approach is based on a
sufficient condition for $c$-compactness suggested in~\cite{DU}, which
leads to the notion of a hereditarily non-topologizable group introduced
by Luk\'acs \cite{Luk}.

Recall that an abstract group is called {\it topologizable} if it admits a
non-discrete Hausdorff group topology, and \emph{non-topologizable}
otherwise. In 1946, A.~A.~Markov \cite{Mar} asked whether there exist
non-topologizable infinite groups and the problem remained open until late
70's. In~\cite{She}, Shelah constructed first (uncountable) examples using
the Continuum Hypothesis. Later Hesse \cite{Hes} showed that the use of
the CH in Shelah's proof can be avoided. The affirmative answer to the
Markov's question for countable groups was obtained by the second author
in~\cite{Ols80} (see also \cite[Theorem 31.5]{book}); the proof uses the
group constructed by Adjan in \cite{A} and is essentially elementary
modulo the main theorem of \cite{A}. Since then many other examples of non-topologizable
groups have been found (see, for example, \cite{KT}).

A group $G$ is called \emph{hereditarily non-topologizable} if for every
$H\le G$ and every $N\lhd H$, the quotient group $H/N$ is
non-topologizable. It is easy to prove  that every
hereditarily non-topologizable group is $c$-compact
with respect to the discrete topology (see \cite[Corollary 5.4]{DU});
moreover, a countable group is hereditarily non-topologizable if and only
if it is $c$-compact with respect to the discrete
topology \cite[Theorem 5.5]{DU}.

Using techniques developed in \cite{book} we prove the following result
(see Theorem \ref{hnt}), which completely solves Problem \ref{prob1}.

\begin{thm}\label{hnt-simple}
There exist hereditarily non-topologizable
(and hence $c$-compact
with respect to the discrete topology)
groups $G$, $H$, $I$, and $J$ such that:
\begin{enumerate}
\item[\rm(a)] $G$ is infinite, finitely generated, and of bounded exponent;

\item[\rm(b)] $H$ is finitely generated and of unbounded exponent;

\item[\rm(c)] $I$ is countable, but not finitely generated;

\item[\rm(d)] $J$ is uncountable.
\end{enumerate}
\end{thm}

On the other hand, it is worth noting that neither of the groups
constructed in \cite{KT,Ols80,She} is $c$-compact (see Remark
\ref{not-c-comp}).

The finitely generated groups $G$ and $H$ from Theorem \ref{hnt-simple}
are the so-called Tarski Monsters, i.e.,
infinite
simple groups with all proper subgroups finite cyclic. First examples of
such groups were constructed by the second author in \cite{Ols79}. Clearly
every non-topologizable Tarski Monster is hereditarily non-topologizable.
This raises the natural question of whether a Tarski Monster can be
topologized. The standard way of defining a non-discrete topology using
chains of subgroups obviously fails for Tarski Monsters. Moreover, most
groups which are known to be topologizable, such as
infinite
residually finite
groups,
infinite
locally finite groups \cite{Bel}, or groups containing infinite
normal solvable subgroups \cite{Hes}, are located on the opposite side of
the group-theoretic universe.

The first (and the only
known)
examples of topologizable Tarski Monsters were
constructed by Morris and Obraztsov in \cite{MO} using methods from
\cite{book}. An essential feature of the Morris--Obraztsov construction is
that their groups have unbounded exponent and for finite exponent their
method of defining a non-discrete topology seems to fail. This motivated
the following.

\begin{q}{\rm {\cite[Question 3]{MO}}}
Does there exist a topologizable quasi-finite group of finite exponent?
\end{q}
Recall that a group is \emph{quasi-finite} if all its proper subgroups are
finite and is of \emph{finite exponent $n$} if $g^n =1$ for some positive integer $n$ and every $g\in G$. In this paper we answer the Morris--Obraztsov's question
affirmatively.

\begin{thm}\label{ttm}
For every sufficiently large odd $n\in \mathbb N$
there exists a topologizable Tarski Monster of exponent $n$.
\end{thm}

Our proof of Theorem \ref{ttm} utilizes the notion of a generic
property in a topological space.  Recall that a subset $S$ of a topological
space $X$ is called a \emph{$G_\delta$ set} if $S$ is an intersection of a
countable collection of open sets. Further one says that a \emph{generic
element of $X$ has a certain property $P$} (or $P$ is \emph{generic} in
$X$) if $P$ holds for every $x$ from some dense $G_\delta$ subset of $X$.
The Baire Category Theorem implies that in a complete metric space the
intersection of any countable collection of dense $G_\delta$ sets is again
dense $G_\delta$. Thus we can combine generic properties: if every property from a countable collection
$\{ P_1, P_2, \ldots \}$ is generic, then so is the whole collection (i.e., the conjunction of $P_1,P_2,\ldots$). In many situations this approach is
useful for proving the \emph{existence} of elements of $X$ simultaneously
satisfying $P_1, P_2, \ldots $.

To implement this idea we consider the \emph{space of marked $k$-generated
groups}, $\mathcal G_k$, which is a compact totally disconnected metric
space consisting of all $k$-generated groups with fixed generating sets.
For the precise definition we refer to Section 3. The study of generic properties in subspaces of $\mathcal G_k$ was initiated by Champetier in \cite{Ch}. The following observation is crucial for our proof of Theorem \ref{ttm}.

\begin{prop}\label{Gd-simple}
For every $k\in \mathbb N$, the following subsets of $\mathcal G_k$
are $G_\delta$:
\begin{enumerate}
\item[\rm(a)] the set of all topologizable groups;
\item[\rm(b)] the set of all Tarski Monsters of any fixed finite exponent.
\end{enumerate}
\end{prop}

Using methods from the book \cite{book}, for every sufficiently large odd
$n\in \mathbb N$ we construct a compact
nonempty subset
$\mathcal T\subseteq \mathcal G_2$ consisting of groups of exponent $n$
such that $\mathcal T$ contains a dense subset of topologizable groups and
a dense subset of Tarski Monsters. Then by Proposition \ref{Gd-simple} the
properties of being topologizable and being a Tarski Monster are generic
in $\mathcal T$. Hence topologizable Tarski Monsters (of exponent $n$) are
generic in $\mathcal T$. In particular, they exist.

All Tarski Monsters discussed above, as well as many other groups with
``exotic" properties, are limits of hyperbolic groups.
It is not difficult to see that every infinite hyperbolic group is
topologizable, but a much weaker condition also makes a group
topologizable; namely,
in the last section
we observe that being topologizable is a generic property among limits of
``hyperbolic-like" groups. More precisely, we consider the class of acylindrically hyperbolic groups introduced in \cite{Osi13}. This class contains all non-elementary hyperbolic groups,
non-elementary relatively hyperbolic groups with proper peripheral
subgroups (e.g., all non-trivial free products other than
$\mathbb Z_2\ast \mathbb Z_2$), mapping class groups of surfaces of genus $>1$, $Out(F_n)$ for $n\ge 2$,
and many other interesting examples. For the
definition and more details we refer to \cite{Osi13} and references therein.

Given a subset
$S\subseteq \mathcal G_k$, we denote by $\overline{S}$ its closure in
$\mathcal G_k$.

\begin{thm}\label{hypemb}
Let $\mathcal{S}\subseteq \mathcal G_k$ be a subset consisting of (marked
$k$-generated) acylindrically hyperbolic groups. Then being topologizable is a generic property in the set
$\overline{\mathcal S}$.
\end{thm}

The proof of Theorem \ref{hypemb} is accomplished by proving that every acylindrically hyperbolic group is topologizable (see Lemma \ref{he}). Then Proposition \ref{Gd-simple} yields the claim.
Finally we sketch a possible application of Theorem \ref{hypemb} to
constructing non-discrete groups with all nontrivial elements conjugate
(for details and motivation see Section~5).

\paragraph{Acknowledgments.} We would like to thank Gabor Luk\'acs and the referee for useful comments.
\section{Hereditarily non-topologizable groups}

Recall that a subset $V$ of a group $G$ is called
\emph{elementary algebraic} if there exist
${a_1,\ldots, a_k \in G}$
and $\e_1, \ldots, \e_k\in \mathbb Z$ such that $V$ is the set of all
solutions of the equation
$
a_0x^{\e_1}a_1x^{\e_2} \cdots a_kx^{\e_k} =1
$
in $G$.

\begin{lem}[A.A. Markov \cite{Mar}]\label{Markov}
A countable group $G$ is non-topologizable if and only if $G\setminus \{1\}$
is a finite union of elementary algebraic sets.
\end{lem}

We say that a group $G$ is given by a \emph{presentation over a free
product ${G_0*G_1*\dots}$} if $G$ is presented in the form
$$
G=(G_0*G_1*\dots)/\nc{R_1,R_2,\dots},
$$
where $\nc{R_1,R_2,\dots}$ denotes the minimal normal subgroup of the free product $G_0*G_1*\dots$ containing $R_1,R_2,\dots$. We are interested in presentations satisfying
\emph{condition \rm $R$} introduced
in
\cite{book}.
The exact definition of this condition is rather technical and will not be used in our paper. For our purpose it suffices to know that R is a condition on the additional relators
$R_1,R_2,\dots$, which allows one to apply the machinery from the book
\cite{book}.
One of these applications is the following theorem.

\begin{ObraztsovTh}
{\rm \cite{Ob89} (see also {\cite[Theorem 35.1]{book}})}.
There exists $N\in \mathbb N$ such that for any odd $n_0\ge N$ and any countable (finite or infinite) family of nontrivial countable
groups $G_0,G_1,\dots$ without elements of order $2$, there is an infinite simple group $O(G_0,G_1,\dots)$
such that the following conditions hold.
\begin{enumerate}
\item[(a)] $O(G_0,G_1,\dots)$ contains all $G_i$ as distinct maximal subgroups.
\item[(b)] Any two distinct maximal subgroups of
$O(G_0,G_1,\dots)$ intersect trivially.
\item[(c)] Any proper subgroup of $O(G_0,G_1,\dots)$ is either cyclic of order dividing $n_0$ or
conjugate to a subgroup of some $G_i$.
\end{enumerate}
Moreover, if our collection $G_0,G_1,\dots$ contains
at least two groups, we also have the following.
\begin{enumerate}
\item[(d)] $O(G_0,G_1,\dots)$ is generated by any pair of non-trivial elements $x,y$ satisfying $x\in G_i$ and $y\notin G_i$ for some $i$.
\item[(e)] $O(G_0,G_1,\dots)$ has a presentation
satisfying condition {\rm R} over the free product $G_0*G_1*\dots$.
\end{enumerate}
\end{ObraztsovTh}

The proof of the next lemma essentially uses the machinery developed in the book \cite{book}. Even brief definitions of all notions used below (condition R, reduced diagram, numerical parameters, etc.) would take many pages, so we choose not to explain them here and simply refer to \cite{book}.

\begin{lem}\label{anticommutator}
If a group $U$ has a presentation with condition {\rm R} over a free
product $G_0*G_1*\dots$, where $G_i$ are countable nontrivial
groups without elements of order $2$, then for any elements $g_0\in G_0\setminus\1$ and
$u\in U\setminus G_0$, the element $g_0g_0^u$ is not conjugate to any
element of any group $G_i$.
\end{lem}

Henceforth, $g^h$ means $h^{-1}gh$ if $g$ and $h$ are elements of a group.

\begin{proof}
Suppose that $g_0g_0^u$ is conjugate to an element $h\in G_i$.
Let
${X\in G_0*G_1*\dots}$ be a word
(in the alphabet $G_0\cup G_1\cup\dots$)
representing $u$.

It follows
that there is a diagram of conjugacy $\Sigma $ of the word $g_0X^{-1}g_0X$
and
the letter $h$. That is, $\Delta $ is an annular van Kampen diagram over the presentation of $U$ with labels of the boundary components equal $g_0X^{-1}g_0X$ and $h$ (read in the appropriate direction). One can identify the subpaths of the
boundary of $\Sigma$ labeled by $X$ and $X^{-1}$ and obtain a diagram~$\Delta_0$ on a
sphere with $3$ holes. Its boundary components
$p_1$, $p_2$ and $q$
are
labeled (e.g., in clockwise manner) by the letters $g_0$, $g_0$, and
$h$, respectively. A simple path  $x$ labeled by $X$ connects the origins  of the
paths $p_1$ and $p_2$.

If $\Delta_0$ is not a reduced diagram, then one can make reductions
described in Section~13
of~\cite{book}
and obtain a reduced diagram
$\Delta$ with the same boundary labels. Moreover, there is a simple path
in $\Delta$ connecting the origins of $p_1$ and $p_2$, whose label is
equal to $X$ in~$U$.

We obtain a
reduced
diagram $\Delta$ on a sphere with 3 holes, but
the length of every its boundary component is equal to 1 (in the
metric of diagrams over free products). It follows that the rank of
$\Delta$ is 0. Indeed, Lemmas 33.3 -- 34.2
\cite{book}
extend the theory of presentations with the condition $R$ to the
presentations over free products; and so if $r(\Delta)>0$, then the length
of one of the boundary components of $\Delta$ must be $>\varepsilon n >1$
by Theorem 22.2.
Here $\epsilon$ and $n$ are some parameters from \cite{book} satisfying
$n\gg{1\over\epsilon}\gg1$; their exact values are not essential for us. (For more details about parameters see Section 4.)

Thus the word $X$ is equal in $U$ to a word $X'$ such that the
element $g_0g_0^{X'}$ is conjugate to an element of $G_i$ in the free
product $G_0*G_1\dots$. But this can happen only if
$X'$
represents an element of $G_0$. Hence $u\in G_0$.
This contradiction completes the proof.
\end{proof}

\begin{thm}\label{emb}
For any sufficiently large integer $n_0$ and any countable
nonempty family of nontrivial countable groups
$G_1,G_2,\dots$ without elements of order $2$, there
exists an infinite simple group
$O'(G_1,G_2,\dots)$ such that the following conditions hold.
\begin{enumerate}
\item[(a)] $O'(G_1,G_2,\dots)$  contains all $G_i$ as distinct maximal subgroups.
\item[(b)] Any two distinct maximal subgroups of
$O'(G_1,G_2,\dots)$  intersect trivially.
\item[(c)] Any proper subgroup of $O'(G_1,G_2,\dots)$  is either cyclic of order at most $n_0$ or
conjugate to a subgroup of some $G_i$.
\item[(d)] Any pair of non-trivial elements $x,y$ satisfying $x\in G_i$ and $y\notin G_i$ for some $i$ generates $O'(G_1,G_2,\dots)$.
\item[(e)] There exists an equation with
one unknown having precisely one non-solution in $O'(G_1,G_2,\dots)$. In particular, $O'(G_1,G_2,\dots)$ is non-topologizable.
\end{enumerate}
\end{thm}

\begin{proof}
Let us take a sufficiently large odd number $n_0$ as in Obraztsov's Theorem and a finite cyclic
group $G_0=\gp{g}$ of odd order coprime to $n_0$ (of order three, for
instance). Let
$$
O'(G_1,G_2,\dots)=O(G_0,G_1,G_2,\dots).
$$

Parts (a)-(d) of the theorem follow immediately from the corresponding parts of Obraztsov's theorem. Thus we only need to prove (e).
To construct an equation with precisely one non-solution
in
$O'(G_1,G_2,\dots)$, we take any element
$v\in O'(G_1,G_2,\dots)\setminus G_0$ and consider the equation
\begin{equation}\label{equat}
[(gg^x)^{n_0},(g^vg^{vx})^{n_0}]=1.
\end{equation}

Let us first show that the identity element is not a solution to this equation. Indeed assume that it is. Note that the equality $[g^{2n_0},g^{2n_0v}]=1$ implies $[g,g^v]=1$, because the order of $g$ is coprime to $2n_0$.  Thus the
group~$\gp{g,g^v}$ is abelian. In particular, $\gp{g,g^v}$ is a proper subgroup of $O'(G_1,G_2,\dots)$. Since $G_0=\gp g$ is maximal in $O'(G_1,G_2,\dots)$ by Obraztsov's theorem, we have $\gp{g,g^v}=\gp g$. This in turn implies that $\gp{g,v}$ is metabelian. Again $\gp{g,v}$ is a proper subgroup of $O'(G_1,G_2,\dots)$ (for example, because the latter group is simple) and hence  $\gp{g,v}=G_0$ by maximality of $G_0$. However this contradicts our choice of $v$.

We note that if an element $x$ does not belong to
$G_0$, then by Lemma~\ref{anticommutator} and part (e) of Obraztsov's theorem the element $gg^x$ is
not conjugate to an element of any $G_i$. Therefore, $gg^x$ generates
a cyclic subgroup of order dividing $n_0$ by Obraztsov's theorem. Hence for every $x\notin G_0$, the first argument of the commutator in (\ref{equat}) is $1$. For the same reason, the second argument of the commutator is $1$ whenever $x\notin (G_0)^v$. Observe that $G_0\cap (G_0)^v=\{ 1\}$, since both subgroups are maximal and distinct by Obraztsov's theorem. Thus all nonidentity
elements of~$O'(G_1,G_2,\dots)$ are solutions to (\ref{equat}). In particular, $O'(G_1,G_2,\dots)$ is non-topologizable
by Lemma~\ref{Markov}.
\end{proof}

\begin{rem}
It is known that
\begin{itemize}
\item
any group embeds into a non-topologizable
group \cite{Trof-emb};

\item
there exists a (non-topologizable)
torsion-free group  of any infinite cardinality
such that some equation has exactly one non-solution in this group
\cite{KT};

\item
there exists an infinite (non-topologizable)
group naturally isomorphic to its automorphism group
such that some equation has exactly one non-solution in this group
\cite{Trof-perf}.

\end{itemize}
\end{rem}

\begin{thm}\label{hnt}
There exist infinite hereditarily non-topologizable simple
torsion
groups $G$, $H$, $I$, and $J$ such that:
\begin{enumerate}
\item[\rm(a)] $G$ is
2-generated, quasi-cyclic,
and of bounded exponent;

\item[\rm(b)]
$H$ is
2-generated, quasi-cyclic,
and of unbounded exponent;

\item[\rm(c)]
$I$ is
not finitely generated, countable,
and of bounded exponent;

\item[\rm(d)]
$J$ is
uncountable
and of bounded exponent.
\end{enumerate}
\end{thm}

\begin{proof}
Let $n_0$ be a sufficiently large odd integer as in Theorem \ref{emb}. Let
$$
G=O'(\Z_{n_0}),
\qquad
H=O'(\Z_{n_0},\Z_{n_0+2},\Z_{n_0+4},\dots),
\qquad
I=\bigcup_{i=0}^\infty G^i,
$$
where $G^0=G_0=\gp g$ is a finite cyclic group of odd
order coprime to $n_0$,
and $G^{i+1}=O(G^i,\Z_{n_0})$ for $i>0$.

The uncountable group $J$ (of the first uncountable cardinality) is
constructed similarly to~$I$ but using the transfinite induction
(up to the first uncountable ordinal).

The groups $G$ and $H$ are quasi-cyclic and simple by Theorem~\ref{emb}.
The groups $I$ and $J$ are simple, because they are unions of increasing
chains of simple subgroups.

All four groups are non-topologizable.
For $G$ and $H$, this follows directly from Theorem~\ref{emb};
$I$ and $J$ are non-topologizable, since
the set of
non-solutions of the equation~$(gg^x)^{n_0}=1$ is nonempty and
finite (it is contained in~$G_0$) by Lemma~\ref{anticommutator}.

Let us show that the groups $I$ and $J$ contain no proper infinite subgroups except for
subgroups conjugate to
$G^i$. Indeed let $P< I$ be a proper infinite subgroup. If $P\le G^k$ for some $k$, part (c) of Theorem \ref{emb} implies (by induction) that $P$ is conjugate to $G^i$ for some $i\le k$. Thus it suffices to rule out the case when for every $k\in \mathbb N$, $P$ contains an element that does not belong to $G^k$. Fix any $k_0\in \mathbb N$. There exists $k_1\ge k_0$ such that $P$ contains a non-trivial element $x\in G^{k_1}$. By our assumption there also exists $k_2> k_1$ such that $P$ contains an element $y\in G^{k_2}\setminus G^{k_2-1}$.  Now part (d) of Theorem \ref{emb} implies that the subgroup $\gp {x,y}$ coincides with $G^{k_2}$. Since $G^{k_0}\le G^{k_1}\le G^{k_2}$, we obtain that $G^{k_0}\le P$. As this holds true for any $k_0\in \mathbb N$, we have $P=I$, which contradicts properness of $P$. This completes the proof for $I$. For the group $J$, the proof is analogous but one has to use transfinite induction instead of the standard one.

Since $G^i$ is simple and non-topologizable for every $i$, it follows that $I$ and $J$ are hereditary
non-topologizable.
\end{proof}

\begin{rem}\label{not-c-comp}
Note that neither of the groups constructed in \cite{KT,Ols80,She} is
$c$-compact. Indeed it is immediate from the definition that
$c$-compactness is preserved by taking closed subgroups (i.e., any
subgroups in the discrete case). Recall that a discrete countable group is
$c$-compact if and only if it is hereditarily non-topologizable. Since
(discrete) groups from \cite{She} and \cite{KT} contain infinite cyclic subgroups,
they are not $c$-compact.

The countable non-topologizable group constructed in \cite{Ols80} is also
not hereditarily non-topologizable, since it has the free Burnside group
$B(m,n)$ with $m\ge 2$ generators and of large odd exponent $n$ as a
quotient. The latter group admits a non-discrete topology defined by a
nested chain of normal subgroup. This can be extracted from \cite[Theorem
39.3]{book}; for $m=2$ this also follows from  Corollary  \ref{fin}
applied to $J=\emptyset$. Alternatively one can argue as follows. By \cite[Theorem 39.1]{book} the
group $B(m,n)$ contains $B(\infty,n)$. Passing to the abelianization we
obtain a countably infinite sum of copies of $\mathbb Z/n\mathbb Z$, which
is obviously topologizable.
\end{rem}

\section{The space of marked groups and $G_\delta$ sets}

Let $F_k$ be the free group of rank $k$ with basis $X=\{x_1, \ldots,
x_k\}$ and let $\mathcal G_k$ denote the set of all normal subgroups of
$F_k$. Given $M,N\lhd F_k$, let
$$
\d(M,N)=\left\{
\begin{array}{ll}
\max
\left\{ \left.\frac1{|w|}\;\right|\; w\in N\vartriangle M
\right\},
& {\rm if} \; M\ne N\\
0,& {\rm if}\; M=N,
\end{array}
\right.
$$
where $|\cdot|$ denotes the word length with respect to the generating set
$X$. It is easy to see that $(\mathcal G_k, \d)$ is a compact Hausdorff
totally disconnected (ultra)metric space \cite{Gri}.

Note that one can naturally identify $\mathcal G_k$ with the set of all
\emph{marked $k$-generated groups}, i.e., pairs
$(G, (x_1,\ldots, x_k))$, where $G$ is a group and $(x_1, \ldots , x_k)$
is a generating $k$-tuple of $G$. (By abuse of notation, we keep the same
notation for the generators $x_1, \ldots , x_k$ of $F_k$ and their images
in $G$.)  For this reason the space $\mathcal G_k$ with the metric defined
above is called the \emph{space of marked groups with $k$ generators}. For
brevity, we simply call elements of $\mathcal G_k$ groups instead of
marked $k$-generated groups.

Let $\mathcal L_k$ be the first order language that contains the standard
group operations $\cdot$, $^{-1}$, the constant symbol $1$, and constant
symbols $x_1, \ldots , x_k$. Every element
$(G, (x_1, \ldots , x_k))\in \mathcal G_k$ can be naturally thought of as
an $\mathcal L_k$-structure.

The following lemma is obvious.

\begin{lem}\label{eq}
Let $w$ be a word in the alphabet $X\cup X^{-1}$. Then for every $k\in
\mathbb N$, the set of groups in $\mathcal G_k$ satisfying $w=1$
(or $w\ne 1$) is clopen.
\end{lem}
\begin{proof}
If $w=1$ (or $w\ne 1$) in a group
$(G, (x_1, \ldots , x_k))\in \mathcal G_k$ and $w$ has length $r$,
then $w=1$ (respectively, $w\ne 1$) in every other group
$(H, (x_1, \ldots , x_k))\in \mathcal G_k$ such that $\d (G,H)< 1/r$.
Thus the set of groups satisfying $w=1$ (respectively, $w\ne 1$) is open
and the claim of the lemma follows.
\end{proof}

Recall that a sentence in a first order language is called an
$\forall\exists$\emph{-sentence} if it has the form
\begin{equation}\label{AE}
\forall a_1\,\ldots\, \forall a_m \,
\exists b_1\,\ldots\, \exists b_n\;
\Phi (a_1, \ldots , a_m, b_1, \ldots, b_n),
\end{equation}
where $\Phi (a_1, \ldots , a_m, b_1, \ldots, b_n)$ is a quantifier-free
formula. If such a sentence only contains existential (respectively,
universal) quantifiers, it is called \emph{existential} (respectively,
\emph{universal}). We say that a subset $\mathcal S\subseteq \mathcal G_k$
is \emph{$\forall\exists $-definable} if there exists an
$\forall\exists$-sentence $\Sigma$ in $\mathcal L_k$ such that
$$
\mathcal S=\{ P\in \mathcal G_k\mid P \models \Sigma\},
$$
i.e., $\mathcal S$ is exactly the set of all elements of $\mathcal G_k$
satisfying $\Sigma$. Similarly we define \emph{existentially definable}
and \emph{universally definable} subsets.

Observe that if $(G, (x_1, \ldots , x_k))\in \mathcal G_k$, then we know
that $x_1, \ldots , x_k$ generate $G$. This allows us to use the following
quantifier elimination procedure. Let $R(u)$ be a (not necessarily first
order) property of marked $k$-generated groups which depends on some
parameter $u$ interpreted as a group element. Enumerate all words $\{
w_1,w_2,\ldots\}$ in the alphabet $X\cup X^{-1}$. Then we obviously have
\begin{equation}\label{elimA}
\{ P\in \mathcal G_k \mid P\models \forall u\, R(u)\} =
\bigcap\limits_{i=1}^\infty \{ P\in \mathcal G_k \mid P\models R(w_i)\}
\end{equation}
and
\begin{equation}\label{elimE}
\{ P\in \mathcal G_k \mid P\models \exists u\, R(u)\} =
\bigcup\limits_{i=1}^\infty \{ P\in \mathcal G_k \mid P\models R(w_i)\}.
\end{equation}

The first part of the following lemma is well-known although we were
unable to find an exact reference.

\begin{prop}\label{exist}
\begin{enumerate}
\item[\rm(a)]{\rm [Folklore]}\;
Every existentially defined subset of $\mathcal G_k$ is open.
\item[\rm(b)]
Every $\forall\exists $-definable subset of $\mathcal G_k$ is a
$G_\delta$ set.
\end{enumerate}
\end{prop}

\begin{proof}
Every existential sentence is equivalent to a sentence
\begin{equation}\label{sigma}
\exists b_1\,\cdots\, \exists b_n\; \Phi_1 (b_1, \ldots , b_n)\, \vee\,
\ldots \vee\, \Phi_q (b_1, \ldots , b_n),
\end{equation}
such that each $\Phi_i$ is a system of equations and inequations of the
form $w=1$ (respectively, $w\ne 1$), where $w$ is a word in the alphabet
$\{x_1^{\pm 1},\ldots,x_k^{\pm1}\}\cup\{b_1^{\pm 1},\ldots,b_n^{\pm 1}\}$.
Thus the first claim follows from Lemma \ref{eq} and the quantifier
elimination (\ref{elimE}) applied to all quantifiers in (\ref{sigma}). To
prove (b) we have to eliminate all universal quantifiers in (\ref{AE})
according to (\ref{elimA}) and apply~(a).
\end{proof}

It would be interesting to find other sufficient conditions
in the spirit of \cite{Mal} and  \cite{Cle} for a
(not necessarily first order) sentence to define a $G_\delta $ subset of
$\mathcal G_k$. In particular, we ask the following.
\begin{q}
Which second order sentences define $G_\delta $ subsets of $\mathcal G_k$?
\end{q}

The next proposition provides some particular non-trivial examples of
$G_\delta $ subsets of $\mathcal G_k$, which are relevant to our paper.
Recall that by a {\it Tarski Monster} we mean a finitely generated
infinite simple group with all proper subgroups finite cyclic.

\begin{prop}\label{Gd}
For every $k\in \mathbb N$, the following subsets of
$\mathcal G_k$ are $G_\delta$:
\begin{enumerate}
\item[\rm(a)] The set of all topologizable groups.
\item[\rm(b)] The set of all infinite groups.
\item[\rm(c)] The set of all groups satisfying a given identity.
\item[\rm(d)] The set of all simple groups.
\item[\rm(e)] The set of Tarski Monsters of any fixed finite exponent.
\item[\rm(f)] The set of groups with all non-trivial elements conjugate.
\end{enumerate}
\end{prop}

\begin{proof}
Let $E=\{ E_{1}, E_{2}, \ldots \}$ denote the set of all finite
collections of equations over the free group $F_k$ with one unknown.
Given an element $(G, (x_1, \ldots, x_k))\in \mathcal G_k$ (i.e., an
epimorphism $\e\colon F_k\to G$), we can think of each $E_n$ as a
collection of equation over $G$ by projecting all coefficients to $G$ via
$\e$. Consider the following condition:
\begin{itemize}
\item[{$\bold C_n$}:]
\emph{Some nontrivial element of $G$ satisfies neither of the equations
from $E_{n}$ or $1$ satisfies at least one of the equations from $E_{n}$.}
\end{itemize}
Clearly every $\bold C_n$ can be expressed by an existential formula in
$\mathcal L_k$. Hence the set $\mathcal C_n$ of elements of
$\mathcal G_k$ satisfying $\bold C_n$ is open for every $n$ by
Proposition \ref{exist}. By Lemma \ref{Markov}, the set of all
topologizable groups in $\mathcal G_k$ coincides with
$\bigcap_{n\in \mathbb N} \mathcal C_n$ and hence it is a $G_\delta$ set
by definition.

To prove (b) we first observe that the set of all groups of order $\le m$
in  $\mathcal G_k$ is finite, hence the set $\mathcal I_m$ of groups
having more than $m$ elements is open. Consequently, the set of all
infinite groups is a $G_\delta$ set being the intersection of all
$\mathcal I_m$.

Part (c) follows from part (b) of Proposition \ref{exist} and the obvious
fact that the subset of $\mathcal G_k$ consisting of groups satisfying
a given identity can be defined by a universal sentence.

Let us prove (d). Fix some word $w$ in $X\cup X^{-1}$ and enumerate all
words $\{ u_1, u_2, \ldots \}$ in the normal closure of $w$ in $F_k$.
Observe that the property
\begin{itemize}
\item[$\bold D_w$:]
\emph{The normal subgroup of $G$ generated by $w$ is trivial or coincides
with $G$}
\end{itemize}
can be expressed by the (infinite) disjunction of formulas
\begin{equation}\label{normsub}
x_1=u_{i_1}  \,\&\,   \ldots \,\&\, x_k=u_{i_k}
\end{equation}
for $\{i_1, \ldots, i_k\} \in \mathbb N^k$ and $w=1$. Hence the set
$\mathcal D_w$ of elements of $\mathcal G_k$ satisfying $\bold D_w$ is the
union of open subsets of $\mathcal G_k$ by Lemma \ref{eq}. Consequently
$\mathcal D_w$ is open. It is easy to check that a group
$(G, (x_1, \ldots, x_k))\in \mathcal G_k$ is simple if and only if it
belongs to $\bigcap_{w\in F_k} \mathcal D_w$. Thus we obtain (d).

The proof of (e) is similar. Fix some $n\in \mathbb N$. For two words
$u,v$ in $X\cup X^{-1}$, consider the set $\mathcal S_{u,v}$ of all
$(G, (x_1, \ldots , x_k))\in \mathcal G_k$ such that the subgroup of $G$
generated by $\{ u,v\}$ is contained in a cyclic subgroup of order
dividing $n$. It is easy to see that $\mathcal S_{u,v}$ can be defined by
an existential formula in $\mathcal L_k$. E.g., for $n=2$ the following
formula works:
$$
\exists z\; (z^2=1\, \&\, ((u=1 \, \&\, v=1) \,\vee\, (u=1 \, \&\, v=z) \,\vee\, (u=z \, \&\, v=1)\,\vee\, (u=z \, \&\, v=z))).
$$
Thus $\mathcal S_{u,v}$ is open in $\mathcal G_k$.

Further let $\{ w_1, w_2, \ldots \}$ be the set of all elements of the
subgroup of $F_k$ generated by $u$ and $v$. For any
$(i_1, \ldots, i_k)\in
\mathbb N^k$, denote by $\mathcal R_{u,v,i_1,\ldots, i_k}$ the set of all
elements of $\mathcal G_k$ satisfying
$$
x_1=w_{i_1} \, \&\, \ldots \, \&\, x_k=w_{i_k}.
$$
By Lemma \ref{eq}, every
$\mathcal R_{u,v,i_1,\ldots, i_k}$ is also open. Thus the set
$$
\mathcal Q_{u,v}=
\mathcal S_{u,v} \cup
\left(
\bigcup\limits_{(i_1, \ldots, i_k)\in \mathbb N^k}
\mathcal R_{u,v,i_1,\ldots, i_k}
\right)
$$
is open and hence the set
$$
\mathcal T_0=\bigcap\limits_{u,v\in F_k} {\mathcal Q}_{u,v}
$$
is a $G_\delta $ set. It is easy to see that $\mathcal T_0$ has the property:
\begin{itemize}
\item[{\bf T}:]
\emph{For every $(G, (x_1, \dots, x_k))\in \mathcal T_0$, every
$2$-generated  subgroup of $G$ is either cyclic of order dividing $n$ or
coincides with $G$.}
\end{itemize}

Let
$$
\mathcal T=\mathcal T_0 \cap \mathcal I\cap
\mathcal S,
$$
where $\mathcal I$ is the set of all infinite groups
and $\mathcal S$ is
the set of all simple groups in $\mathcal G_k$. Then $\mathcal T$ is a
$G_\delta$ subset of $\mathcal G_k$. We want to show that $\mathcal T$ is
exactly the subset of all (marked $k$-generated) Tarski Monsters
satisfying the identity $x^n=1$.

Indeed suppose $(G, (x_1, \ldots , x_k))\in \mathcal T$. Let $H$ be a
proper subgroup of $G$. According to {\bf T} every $2$-generated subgroup
of $H$ is cyclic of order dividing $n$. This obviously implies that $H$
itself is cyclic of order dividing $n$. Note also that $G$ is infinite,
simple, and satisfies $x^n=1$ by the definition of $\mathcal T$.
Conversely, it is easy to see that every Tarski Monster satisfying the
identity $x^n=1$ belongs to $\mathcal T$.

Finally to prove (f) it suffices to note that the subset of $\mathcal G_k$
consisting of groups with $2$ conjugacy classes can be defined by the
$\forall\exists$-formula
$$
\forall x\, \forall y\,  \exists t\,
(x=1 \,\vee\, y=1 \,\vee\, t^{-1}xt=y).
$$
Now applaying part (b) of Proposition \ref{exist} finishes the proof.
\end{proof}

\section{Topologizable Tarski Monsters}

Our proof of Theorem \ref{ttm} makes use of a particular variant of the
general construction described in \cite[Sections 25-27]{book}. The variant
used here is similar to that from \cite[Section 39.2]{book}. Below we
briefly recall it and refer the reader to \cite{book} for details.

Given a group $G$ generated by a set $X$, we write ``$A\equiv B$" for two
words in the alphabet $X\cup X^{-1}$ if they coincide as words (i.e.,
letter-by-letter) and ``$A=B$ in $G$" if $A$ and $B$ represent the same
elements of $G$; by abuse of notation we identify words in $X\cup X^{-1}$
and elements represented by them. As in \cite{book}, given a word $A$ in
some alphabet, $|A|$ denotes its length.

The general construction in \cite[Sections 25-27]{book} uses a sequence of
fixed positive small parameters
$$
\alpha, \; \beta,\; \gamma,\; \delta, \;\e,\; \zeta,\; \eta,\; \iota.
$$
The exact relations between the parameters are described by a system of
inequalities, which can be maid consistent by choosing each parameter in
this sequence to be sufficiently small as compared to all previous
parameters. In \cite{book} and below, this way of ensuring consistency is
referred to as the \emph{lowest parameter principle}
(see \cite[Section 15.1]{book}).
Below we will use the following auxiliary parameters (which are assumed to be integers):
$$
h=\delta^{-1}, \; d=\eta^{-1},\; n=\iota^{-1} .
$$
We also fix a sufficiently large odd $n_0\in \mathbb N$ satisfying
\begin{equation}\label{n0}
n_0> \max\left\{ (h+1)n,\, \frac{h (d+n+2h-2)}{1-\alpha}\right\} .
\end{equation}

\begin{rem}
Our notation in this section is borrowed from \cite{book} and is different from the notation in the introduction: the exponent denoted by $n$ in Theorem \ref{ttm} is denoted by $n_0$ here.
\end{rem}

Given a subset $J\subseteq \mathbb N$, we construct groups $G(i,J)$ by
induction on $i\in \mathbb N$ as follows. Let $G(0,J)=F(a_1,a_2)$ be the
free group with basis $\{ a_1, a_2\}$. Suppose now that
$$
G(i-1,J)=\langle a_1,a_2\mid \mathcal R_{i-1}\rangle
$$
is already constructed for some $i\ge 1$, and that for each $1\le j\le
i-1$ we have already defined a set $\mathcal X_j$ of words of length $j$
in $\{a^{\pm 1}, b^{\pm 1}\}$ called \emph{periods of rank $j$}.

The set of periods of rank $i$, $\mathcal X_i$, is defined to be a maximal set of words of length $i$ in the alphabet $\{a^{\pm 1}, b^{\pm 1}\}$ such that no $A\in \mathcal X_i$ is conjugate to a power of a word of length $<i$ in the group $G(i-1,J)$, and if $A$ is conjugate to $B$ or $B^{-1}$ in
$G(i-1,J)$ for some $A,B\in \mathcal X_i$ then $A\equiv B$.

The group $G(i,J)$ is obtained from $G(i-1,J)$ by adding a set of
relations $\mathcal S_i$ constructed as follows. First for each period
$A\in \mathcal X_i$, $\mathcal S_i$ contains the relation
\begin{equation}\label{t1}
A^{n_0}=1
\end{equation}
called a \emph{relation of the first type of rank $i$}.

If $i\notin J$, no other relations are included in $\mathcal S_i$. If
$i\in J$, then for each $A\in \mathcal X_i$ we fix some maximal set of
words $\mathcal Y_A$ such that:
\begin{enumerate}
\item[\rm(a)]
For any $T\in \mathcal Y_A$, we have  $1\le |T|\le d|A|$;
\item[\rm(b)]
Every double coset $\langle A\rangle g\langle A\rangle$ in $G(i-1,J)$
contains at most one word from $\mathcal Y_A$ and this word has minimal
length among all words representing elements of
$\langle A\rangle g\langle A\rangle$  in $G(i-1,J)$.
\end{enumerate}
If $a_1\notin \langle A\rangle$ in $G(i-1)$, then for every $T\in \mathcal
Y_A$ such that $T\notin \langle A\rangle a_1\langle A\rangle $ in $G(i-1,J)$, we add the
relation
\begin{equation}\label{t21}
a_1A^nTA^{n+2}\ldots TA^{n+2h-2}=1
\end{equation}
to the set $\mathcal S_i$. Further if
$a_2\notin  \langle A\rangle\cup \langle A\rangle a_1\langle A\rangle$
and $T\notin \langle A\rangle a_2\langle A\rangle$ in $G(i-1,J)$, then we
also add the relation
\begin{equation}\label{t22}
a_2A^{n+1}TA^{n+3}\ldots TA^{n+2h-1}=1
\end{equation}
to $\mathcal S_i$. These relations are called \emph{relations of the
second type of rank $i$}.

Finally we define
$$
G(i,J)=\langle a_1,a_2 \mid \mathcal R_{i-1}\cup \mathcal S_i\rangle.
$$
Note that there is some freedom in choosing periods in every rank and sets
$\mathcal Y_A$.  We additionally require our construction to be
\emph{uniform} in the following sense: if $I\cap [1,r]=J\cap[1,r]$ for
some $r\in \mathbb N$, then the sets of periods and the corresponding sets
$\mathcal Y_A$ in $G(i,I)$ and $G(i,J)$ coincide for all $1\le i\le r$. In
particular, $G(i,I)$ and $G(i,J)$ have the same relations for all
$1\le i\le r$.

Let $G(\infty, J)$ denote the
limit
group of the sequence $G(0,J)\to  G(1,J)\to  \ldots $. That is,
$$
G(\infty, J)=\left\langle a_1, a_2\;
\left|\; \bigcup\limits_{i=1}^\infty \mathcal S_i\right.\right\rangle .
$$
The presentations of $G(i,J)$, $i\in \mathbb N\cup\{ \infty\} $,
constructed above will be called {\it canonical}.

\begin{rem}
In our notation, the groups $G(i,j)$ constructed in
\cite[Section 39.2]{book} are exactly $G (i, \{j+1, j+2, ...\})$.
\end{rem}

We will need analogues of Lemma 39.5 and Lemma 39.6 from \cite{book}.
Recall that the condition $R$ is a technical condition which allows to
apply the techniques developed in \cite[Sections 25-27]{book}. For the
definition, we refer to \cite[Section 25]{book}.

\begin{lem} \label{Gij}
\begin{enumerate}
\item[\rm(a)]
For every $i\in \mathbb N$ and $J\subseteq \mathbb N$, the presentation of
the group $G(i,J)$ constructed as above satisfies the condition $R$.

\item[\rm(b)]
For every $J\subseteq \mathbb N$, the group $G(\infty, J)$ is infinite and
torsion of exponent $n_0$.

\item[\rm(c)]
If $J$ contains all but finitely many natural numbers, then every
proper subgroup of $G(\infty,J)$ is cyclic of order dividing
$n_0$.
\end{enumerate}
\end{lem}

\begin{proof}
The proof of the first statement almost coincides with the proof of
Lemma 27.2 in \cite{book}. The only difference is that in our construction
we choose $n_0$ to satisfy (\ref{n0}), while in \cite{book} one takes
$n_0$ such that $n=[(h+1)^{-1}n_0]$. However the latter equality is not
essential for the proof of Lemma 27.2. What is really used there is the
inequality $(h+1)n\le n_0$ (see the last line of the proof), which follows
from (\ref{n0}).

Now part (a) allows us to apply Theorems 26.1 and 26.2 from \cite{book},
which yield (b). Finally the proof of (c) repeats the proof of \cite[Lemma
39.6]{book} verbatim after replacing $G(\infty,j)$ with $G(\infty,J)$, and
$j$ with $\max (\mathbb N\setminus J)$. The key point here is that all
relations of the second type of rank $>\max (\mathbb N\setminus J)$ are
imposed in $G(\infty,J)$.
\end{proof}

In the next lemma, we could replace ``arbitrary large" with ``every".
However the weaker statement is sufficient for our goals.

\begin{lem}\label{periods}
For any $J\subseteq \mathbb N$, there exist periods of arbitrary large
rank. That is, for every $r\in \mathbb N$, the set of periods
$\mathcal X_i$ is non-empty for some $i> r$.
\end{lem}

\begin{proof}
We repeat the main argument from the proof of
\cite[Theorem 19.3]{book} with obvious changes. Fix some $r\in \mathbb N$.
By \cite[Lemma 4.6]{book} there exists a $6$-aperiodic word $X$ in the
alphabet $\{ a_1, a_2\}$ of length at least $20r$.  Assume first that
$X^{n_0}=1$ in $G(r, J)$. Arguing as in the second paragraph of the proof
of \cite[Theorem 19.1]{book} (and replacing the reference to
\cite[Theorem 16.2]{book} there with the reference to
\cite[Theorem 22.2]{book}) we conclude that the cyclic word~$X^{n_0}$
contains a subword of the form $A^{20}$ for some non-trivial $A$ of length
at most $r$. Since the length of $X$ is greater than $20r$, this
contradicts the assumption that $X$ is $6$-aperiodic. This contradiction
shows that $X^{n_0}\ne 1$ in $G(r, J)$.  In particular, we have $G(\infty,
J)\ne G(r, J)$  as $X^{n_0}=1$ in $G(\infty, J)$ by part (b) of Lemma
\ref{Gij}. Therefore periods of rank $>r$ exist.

\end{proof}

Every group $G(i,J)$ comes with a natural generating set, namely the image
of $\{ a_1, a_2\} $ under the natural homomorphism $F_2\to G(i,J)$. By
abuse of notation we denote the image of $\{a_1, a_2\}$ in $G(i,J)$,
$i\in \mathbb N\cup \{\infty\}$, by $\{a_1, a_2\}$ as well. In what
follows we say that a homomorphism $\e \colon G(\infty,I)\to G(\infty,J)$
is {\it natural} if $\phi(a_1)=a_1$ and $\phi(a_2)=a_2$.

\begin{lem}\label{nathom}
Let $J\subseteq \mathbb N$ and let $I=J\cap [1,r]$ for some
$r\in \mathbb N$. Then the following hold:
\begin{enumerate}
\item[\rm(a)]
There exists a natural homomorphism
$\e \colon G(\infty, I)\to G(\infty, J)$.
\item[\rm(b)]
${\rm Ker}\, \e$
does not contain nontrivial elements of
$G(\infty, I)$ of length $\le r$ with respect to the generating set
$\{a_1,a_2\}$.
\end{enumerate}
\end{lem}

\begin{proof}
We first note that claim (a) is not obvious as, in general, the set of defining
relations in the canonical presentation of $G(\infty, I)$ is not a subset
of the set of relations in the canonical presentation of $G(\infty, J)$.
However it is possible to construct other presentations of $G(\infty, I)$
and $G(\infty, J)$ for which this is the case.

Let $\mathcal R_I$ and $\mathcal R_J$ be the sets of relations of the
second type in the canonical presentations of $G(\infty, I)$ and
$G(\infty, J)$, respectively. By uniformness of our construction, we have
$\mathcal R_I\subseteq \mathcal R_J$. Since both $G(\infty, I)$ and
$G(\infty, J)$ are torsion of exponent $n_0$ by part (b) of Lemma
\ref{Gij} and all relations of the first type have the form $X^{n_0}=1$
for some word $X$ in the alphabet $\{ a^{\pm 1}_1, a^{\pm }_2\}$, we can
represent the groups $G(\infty, I)$ and $G(\infty, J)$ as follows:
$$
G(\infty, I)=\langle a_1,a_2\mid \mathcal R_I, \; X^{n_0} =1 \;
\forall X\rangle
$$
and
$$
G(\infty, J)=\langle a_1,a_2\mid \mathcal R_J, \; X^{n_0} =1 \;
\forall X\rangle,
$$
where the relations $X^{n_0}=1$ are imposed for all words $X$ in
$\{a^{\pm 1}_1, a^{\pm }_2\}$. Now part (a) of the lemma becomes obvious.

Further part (a) of Lemma \ref{Gij} allows us to apply Lemma 23.16 from
\cite{book}, which implies that every nontrivial element from ${\rm Ker\,}
\e$ has length at least $(1-\alpha)$ times the minimal possible length of
a relator of rank $>r$. It is easy to see from (\ref{t1})-(\ref{t22}),
that the length of every relator of rank $>r$ is at least $(r+1) \min \{
n_0, (2h-1)n\}>rn$. By the lowest parameter principle we can assume that
$(1-\alpha)n>1$. Hence every nontrivial element from ${\rm Ker}\, \e$ has
length at least $r$.
\end{proof}

In what follows we think of $G(\infty,J)$ (or, more precisely, $(G(\infty,
J), \{ a_1, a_2\})$) as an element of $\mathcal G_2$. Let $\mathcal T$ be
the subspace of $\mathcal G_2$ consisting of $G(\infty, J)$ for all
$J\subseteq \mathbb N$. To apply the Baire Theorem to $\mathcal T$ we
need to know that $\mathcal T$ is complete as a metric space. We will
prove this by showing that $\mathcal T$ is a continuous image of the
Cantor set. Recall that the Cantor set $C$ can be identified with
$2^{\mathbb N}$, where the distance between any two distinct subsets
$I,J\subseteq \mathbb N$ is defined by
$$
\d (I,J)=\frac{1}{\min (I \vartriangle J)}.
$$

\begin{cor}\label{Cantor}
The map from the Cantor set $C$ to $\mathcal T$ defined by $J\mapsto
(G(\infty, J), (a_1,a_2))$ is Lipschitz. In particular, this map is
continuous and $\mathcal T$ is compact.
\end{cor}

\begin{proof}
Let $I,J\subseteq \mathbb N$. Suppose now that $\d (I,J)=1/r$ for some
$r\ge 1$ in $C$. Let $K=I\cap [1,r-1]=J\cap[1,r-1]$. By part (b)
of Lemma \ref{nathom}, we have $\d (G(\infty, I), G(\infty, K))\le 1/r$
and $\d (G(\infty, J), G(\infty, K))\le 1/r$. Since $\d$ is an ultrametric, we obtain
$$
\d (G(\infty, I), G(\infty, J))\le \max\{ \d (G(\infty, I), G(\infty, K)), \d (G(\infty, J), G(\infty, K)))\le 1/r= \d (I,J).
$$
Thus the map $C\to \mathcal T$ is $1$-Lipschitz.
\end{proof}

Our next goal is to show that $\mathcal T$ contains a dense subset of
topologizable groups. We begin with an auxiliary result.

\begin{lem}\label{non-simple}
Let $I$ be a finite subset of $\mathbb N$. Then for every non-trivial
element $g\in G(\infty, I)$, there exists a non-trivial normal subgroup
$N\lhd G(\infty, I)$ such that $g\notin N$.
\end{lem}

\begin{proof}
Let $l$ denote the word length of the element $g$ with respect to the
generating set $\{ a_1,a_2\}$. By Lemma \ref{periods}, there exists a
period $A$ of some rank
\begin{equation}\label{i}
i>\max \{l, \max I\} ,
\end{equation}
Since $G(\infty,I)$ is infinite by part (b) of Lemma \ref{Gij}, we can
additionally assume that balls of radius $i$ in $G(\infty,I)$ contain more
than $n_0^2$ elements.

Note that the double coset $\langle A\rangle a_1 \langle A\rangle $ in
$G(\infty,I)$ contains at most $n_0^2$ elements as $A^{n_0}=1$ in
$G(\infty,I)$. Therefore, by our choice of $i$, there exists a word $T$ of
length $1\le |T|\le i< di$ such that $T$ does not belong to
$\langle A\rangle a_1 \langle A\rangle $ in $G(\infty,I)$. Hence $T$ does
not belong to $\langle A\rangle a_1 \langle A\rangle $ in $G(i-1,I)$.
Replacing $T$ with the shortest word among all words representing elements
of the double coset $\langle A\rangle T \langle A\rangle \le G(i-1,I)$ if
necessary, we can assume that $T\in \mathcal Y_A$.

Let now $J=I\cup \{ i\}$. By (\ref{i}), Lemma \ref{nathom} applies to $I$
and $J$ with $r=i-1\ge l$. Let $N$ be the kernel of the natural
homomorphism $G(\infty, I)\to G(\infty, J)$. Then by part (b) of Lemma
\ref{nathom} we have $g\notin N$.

It remains to show that $N$ is nontrivial. To this end, we will show that
$$
1\ne a_1A^nTA^{n+2}\ldots TA^{n+2h-2}\in N
$$
in $G(\infty, I)$. Indeed $a_1A^nTA^{n+2}\ldots TA^{n+2h-2} \in N$ by the
construction of $G(\infty, J)$. Suppose that $a_1A^nTA^{n+2}\ldots
TA^{n+2h-2}=1$ in $G(I, \infty)$.  Let $\Delta $ be the corresponding
reduced disk diagram over $G(I, \infty)$. Then $\Delta $ is a $B$-map by
\cite[Lemma 26.5]{book} and part (a) of Lemma \ref{Gij}.

Note that $\Delta $ does not contain faces of rank $\ge i$. Indeed if such
faces existed, they would correspond to relations of the first type as
there are no relations of the second type of rank $\ge i$ in
$G(I,\infty)$. However by our choice of $n_0$ this is impossible since these
faces are ``too large"; more precisely, by \cite[Lemma 23.16]{book}, the
perimeter of each face in $\Delta $ is at most
$$
\frac{|\partial \Delta |}{1-\alpha} \le \frac{hi(d+n+2h-2)}{1-\alpha}< n_0 i
$$
(see (\ref{n0})), while the length of every relation of the first type of
rank $\ge i$ is at least $n_0i$.

Thus $\Delta $ is a diagram over $G(i-1,I)$. Since $G(i-1,I)=G(i-1,J)$,
$\Delta $ is also  a diagram over $G(i-1,J)$. Hence the relation
$a_1A^nTA^{n+2}\ldots TA^{n+2h-2}=1$ can be derived from relations of rank
$<i$ in $G(\infty,J)$. This contradicts \cite[Corollary 25.1]{book}, which
guarantees that the relations of the canonical presentation of
$G(\infty,J)$ are independent. The contradiction shows that
$a_1A^nTA^{n+2}\ldots TA^{n+2h-2}\ne 1$ in $G(I, \infty)$ and therefore
$N$ is non-trivial.
\end{proof}

\begin{cor}\label{fin}
Let $J$ be a finite subset of $\mathbb N$. Then $G(\infty, J)$ is
topologizable.
\end{cor}

\begin{proof}
It suffices to construct a sequence of infinite normal subgroups
\begin{equation}\label{Ni}
N_1\rhd N_2\rhd \ldots
\end{equation}
of $G(\infty, J)$ with trivial intersection. Then taking
$\{ N_i\}_{i\in \mathbb N}$ as the base of neighborhoods of $1$, we obtain
a group topology on $G(\infty, J)$ which is Hausdorff
as $\bigcap_{i\in \mathbb N} N_i=\{ 1\}$ and is non-discrete as every
$N_i$ is infinite.

To this end, we first note that every non-trivial normal subgroup
$M\lhd G(\infty, J)$ is infinite. Indeed otherwise the centralizer
$C_{G(\infty, J)}(M)$ has finite index in $G(\infty, J)$.
By \cite[Theorem 26.5]{book} the centralizer of every element in
$G(\infty, J)$ is cyclic.  Since $G(\infty, J)$ is torsion by part (b) of
Lemma \ref{Gij} we obtain that $C_{G(\infty, J)}(M)$ is finite and hence
so is $G(\infty, J)$.  However this contradicts part (b) of Lemma \ref{Gij}.

Now we construct the desired sequence (\ref{Ni}) by induction. Let
$G(\infty, J)=\{1, g_1, g_2, \ldots \}$.
By Lemma \ref{non-simple} we can find a non-trivial subgroup $N_1\lhd
G(\infty, J)$ that does not contain $g_1$.  Suppose that $N_j$ is already
constructed for some $j\ge 1$ and $\{ g_1, \ldots , g_{j}\}\cap
N_j=\emptyset$. Applying  Lemma \ref{non-simple} again, we can find a
non-trivial subgroup $N\lhd G(\infty, J)$ such that $g_{j+1}\ne N$. Let
$N_{j+1}=[N_j,N]$. Obviously $\{ g_1, \ldots , g_{j+1}\}\cap
N_{j+1}=\emptyset$.  Note also that $N_{j+1}$ is nontrivial. Indeed
otherwise $N_j\le C_{G(\infty, J)} (N)$ and arguing as in the previous
paragraph we obtain that $N_j$ is finite; however this contradicts the
fact that every non-trivial normal subgroup of $G$ is infinite. This
completes the inductive step. Obviously
$\bigcap_{i\in \mathbb N} N_i=\{1\}$ and thus the lemma is proved.
\end{proof}

\begin{rem}\label{Gd-in-a-subset}
If $D$ is a $G_\delta $ subset of a topological space $X$ and
$Y$ is a subspace of $X$, then $D\cap Y$ is a $G_\delta $ subset of $Y$. This observation will be used several times below.
\end{rem}

\begin{proof}[Proof of Theorem \ref{ttm}]
Note that the set of finite subsets of $\mathbb N$ is dense in the Cantor
set $C$. Hence its image is dense in $\mathcal T$ by Corollary
\ref{Cantor}. Using Lemma \ref{fin} we obtain that $\mathcal T$ contains a
dense subset of topologizable groups. Then by Proposition \ref{Gd} and Remark \ref{Gd-in-a-subset} we
conclude that the property of being topologizable is generic
in $\mathcal T$.

Further the set of all cofinite subsets of $\mathbb N$ is also dense in
the Cantor set. Using Lemma~\ref{Gij} and arguing as in the previous
paragraph, we obtain that the property of being a Tarski Monster (of
exponent $n_0$) is generic in $\mathcal T$.

Since $\mathcal T$ is compact, we can apply the Baire Category Theorem,
which implies that the property of being a topologizable Tarski Monster
(of exponent $n_0$) is also generic in $\mathcal T$. In particular, such
groups exist.
\end{proof}


\section{Further speculations}

One can produce many other examples of ``exotic" topologizable groups
using the fact that most limits of ``hyperbolic-like" groups are
topologizable. More precisely, we recall that an isometric action of a group $G$ on a metric space $S$ is called {\it acylindrical} if for every $\e>0$ there exist $R,N>0$ such that for every two points $x,y$ with $\d (x,y)\ge R$, there are at most $N$ elements $g\in G$ satisfying
$$
\d(x,gx)\le \e \;\;\; {\rm and}\;\;\; \d(y,gy) \le \e.
$$
A group $G$ is called \emph{acylindrically hyperbolic} if it acts acylindrically and non-elementary on a (Gromov) hyperbolic space. Recall also that non-elementarity of the action can be defined in this context by requiring that $G$ is not virtually cyclic and has unbounded orbits. For details we refer to \cite{Osi13}.

The class of acylindrically hyperbolic groups contains many examples of interest: non virtually cyclic hyperbolic groups, non virtually cyclic  relatively hyperbolic groups with proper peripheral subgroups, all but finitely many mapping class groups of punctured closed surfaces, $Out(F_n)$ for $n\ge 2$, groups acting properly on proper $CAT(0)$ spaces and containing rank $1$ elements, and so forth \cite{DGO,Osi13}. 

The proof of the following lemma relies heavily on results of \cite{DGO}; we will refer the reader to \cite{DGO} for definitions of the auxiliary notions used in the proof.  In the particular
cases of hyperbolic and relatively hyperbolic groups one could
alternatively use results of \cite{Ols93} or \cite{Osi07}. 

\begin{lem}\label{he}
Every acylindrically hyperbolic group is topologizable.
\end{lem}

\begin{proof}
Let $G$ be an acylindrically hyperbolic group. By \cite[Theorem 1.2]{Osi13}, $G$ contains non-degenerate hyperbolically embedded subgroups (see \cite{DGO} for the definition). This allows us to apply \cite[Theorem 2.23]{DGO}, which guarantees that there exists a hyperbolically embedded
subgroup $H\le G$ such that $H\cong \mathbb Z\times K$, where $K$ is a
finite group. In particular, $H$ contains an infinite chain of infinite
normal (in $H$) subgroups $N_1\rhd N_2 \rhd \ldots$ with trivial
intersection. Let $M_i$ denote the normal closure of $N_i$ in $G$. Since $H$ is hyperbolically embedded, 
the group-theoretic Dehn surgery theorem (see \cite[Theorem 2.25
(c)]{DGO}) implies that $\bigcap_{i\in \mathbb N} M_i=\{ 1\}$. Now we can
use the chain $M_1\rhd M_2 \rhd \ldots$ to define a Hausdorff
topology on $G$, taking $\{ M_i\mid i\in \mathbb N\}$ as the base of
neighborhoods at $1$. Since every $N_i$ is infinite, so is $M_i$ and hence
the topology is non-discrete.
\end{proof}

Using chains of normal subgroups $N_1\rhd N_2 \rhd \ldots$ as bases of neighborhoods is a fairly standard approach to defining a topology on a given group $G$. It is interesting to ask whether one can topologize a ``hyperbolic-like" group in an essentially different way. The question can be formalized as follows: \emph{Under what conditions does an acylindrically hyperbolic group admit a topology with respect to which it is topologically simple?}

Note that acylindrically hyperbolic groups are very far from being abstractly simple \cite{DGO}. Of course, $G$ is not topologically simple in the topology defined in the proof of Lemma \ref{he} as well, since every $N_i$ is closed. However, we conjecture the following.

\begin{conj}\label{topsim}
Suppose that an acylindrically hyperbolic group $G$ has no non-trivial finite normal subgroups. Then $G$ admits a topology with respect to which it is topologically simple.
\end{conj}

Note that the absence of non-trivial finite normal subgroups is necessary as finite subgroups are always closed.

Conjecture \ref{topsim} holds for hyperbolic groups. Indeed, Chaynikov \cite{Cha} proved that every non-elementary hyperbolic group $G$ without non-trivial finite normal subgroups admits a faithful action on $\mathbb N$ which is $k$-transitive for every $k\in \mathbb N$. This action defines a dense embedding $G\to S(\mathbb N)$, where $S(\mathbb N)$ is the group of all permutations of $\mathbb N$ endowed with the topology of pointwise convergence. Let $A_{fin}(\mathbb N)=\bigcup_{n\in \mathbb N} A_n$ and $S_{fin}(\mathbb N)=\bigcup_{n\in \mathbb N} S_n$, where $A_n$ and $S_n$ are the groups of even permutations and all permutations of $\{ 1, \ldots, n\}$, respectively, naturally embedded in $S(\mathbb N)$. Then $A_{fin}(\mathbb N)$ and $S_{fin}(\mathbb N)$ are the only proper non-trivial normal subgroups of $S(\mathbb N)$ (see \cite[Theorem 8.1A]{DM}). Obviously both of them are dense and hence $S({\mathbb N})$ is topologically simple. Now using the fact that the image of $G$ is dense in $S({\mathbb N})$ it is straightforward to verify that $G$ is topologically simple with respect to the topology induced by the embedding. In seems plausible that the Chaynikov's result can be generalized to groups from $\mathcal H_k$, which would imply Conjecture \ref{topsim} in the full generality.

\begin{proof}[Proof of Theorem \ref{hypemb}]
The theorem obviously follows from Lemma \ref{he}, part (a) of
Proposition~\ref{Gd}, and Remark \ref{Gd-in-a-subset}.
\end{proof}

To illustrate usefulness of Theorem \ref{hypemb}, we outline here the
proof of the existence of a topologizable groups with 2 conjugacy
classes. Details will appear in the forthcoming paper \cite{Hull}.

First examples of groups with 2 conjugacy classes other than
$\mathbb Z/2\mathbb Z$ were constructed by Higman, B.H. Neumann and
H. Neumann in~1949; first finitely generated examples were constructed by
the third author in \cite{Osi}. Motivated by the recent study of groups
with the Rokhlin property, (i.e., topological groups with a dense
conjugacy class) Glassner and Weiss ask in \cite{GW} whether there exist
topological analogues of these constructions. Specifically, they ask
whether there exists a non-discrete locally compact topological group with
2 conjugacy classes. Our approach allows to construct a non-discrete
group with 2 conjugacy classes; local compactness can not be ensured by
our methods although our group will be compactly generated (and even
finitely generated in the abstract sense).

To construct a topologizable group with $2$ conjugacy classes we first
recall that groups with exactly two conjugacy classes form a $G_\delta$
subset of $\mathcal G_k$ by part (f) of Proposition \ref{Gd}. Further let
$k\ge 2$ and let $\mathcal{AH}_{tf}$ denote the subset of all groups from
$\mathcal G_k$ that are torsion free and acylindrically hyperbolic. The technique developed
in \cite{Osi} can be extended to acylindrically hyperbolic groups to show that  $\overline{\mathcal{AH}_{tf}}$
contains a dense $G_\delta $ subset of groups with $2$ conjugacy classes; the proof can be found in \cite[ Corollary 8.10]{Hull}. 
Combining this with Theorem \ref{hypemb}, we obtain that a generic group
in $\overline{\mathcal{AH}_{tf}}$ is topologizable and all its non-trivial elements are conjugate.


\vspace{1cm}

\noindent \textbf{Anton A. Klyachko: } Department of Mechanics and Mathematics, Moscow State University, Russia.\\
E-mail: \emph{klyachko@mech.math.msu.su}\\

\smallskip

\noindent \textbf{Alexander A. Olshanskii: } Department of Mathematics, Vanderbilt University, Nashville 37240, U.S.A.\\
E-mail: \emph{alexander.olshanskiy@vanderbilt.edu}\\

\smallskip

\noindent \textbf{Denis V. Osin: } Department of Mathematics, Vanderbilt University, Nashville 37240, U.S.A.\\
E-mail: \emph{denis.osin@gmail.com}

\end{document}